\documentclass{article}[12pts]
\usepackage{hyperref}

\usepackage{amssymb}
\usepackage{mathtools}
\usepackage{mathdots}

\usepackage{amsmath}
\usepackage{amscd}
\usepackage{amsthm}
\usepackage{graphicx}
\usepackage[all]{xy}

\usepackage{color}          
\usepackage{epsfig}

\usepackage{lineno}

\newcommand{\C}{\mathbb {C}}

\newcommand{\Addresses}{{
  \bigskip
  \footnotesize

  C.~Cabrera, \textsc{Unidad Cuernavaca del Instituto de Matem\'aticas. UNAM, M\'exico}\par\nopagebreak
  \textit{E-mail:}, C.~Cabrera: \texttt{carloscabrerao@im.unam.mx}

  \medskip

  P.~Makienko  \textsc{Unidad Cuernavaca del Instituto de Matem\'aticas. UNAM, M\'exico}\par\nopagebreak
  \textit{E-mail:}, P.~Makienko: \texttt{makienko@im.unam.mx}}}

\newtheorem{theorem}{Theorem}
\newtheorem{lemma}[theorem]{Lemma}
\newtheorem{corollary}[theorem]{Corollary}
\newtheorem{proposition}[theorem]{Proposition}      \newtheorem{conjecture}{Open question}
\theoremstyle{definition}                             
\newtheorem*{definition}{Definition}

\title{On amenability and measure of maximal entropy for semigroups of rational maps: II.}
\author{Carlos Cabrera and Peter Makienko}

\begin{document}

\maketitle

\footnotetext{This work was partially supported by PAPIIT IG100523. MSC2010: 37F10, 43A07}
 \begin{abstract}

We compare dynamical and algebraic properties of  semigroups of rational maps. In particular, we show a version of the Day-von Neumann conjecture and give a partial positive answer to ``Sushkievich's problem''  for semigroups of rational maps. We relate these conjectures with Furstenberg's $\times 2 \times 3$ problem and prove a coarse version of Furstenberg's problem for semigroups of non-exceptional polynomials. 
\end{abstract}

\section{Introduction and statements}

We discuss versions of some classical problems in the context of semigroups of rational maps; these are collections of holomorphic endomorphisms of the Riemann sphere $\overline{\C}$ equipped with composition as a product. Namely, we prove the right amenable versions of the Day-von Neumann conjecture. Also, we confirm Sorensen's conjecture for polynomial semigroups. This implies that a left amenable semigroup of polynomials is embeddable into a group (see Corollary \ref{Cor.leftAmaneble}). Further, we discuss a criterion for the embeddability of a semigroup into a group. In particular, we give a solution to Sushkevich's problem for a class semigroups of rational maps (see Theorem \ref{th.embeddable}). Finally, we relate the previous questions with Furstenberg's $\times 2 \times 3$ problem and show a coarse version of this problem for semigroups of non-exceptional polynomials.   

We start with a version of the Day-von Neumann conjecture for semigroups of rational maps. Let us recall the Day-von Neumann conjecture for groups which was formulated by Day in \cite{DayAmenable}: \textit{A finitely generated group is amenable if and only if it does not contain a free subgroup of rank $2$.}

  This conjecture was proved to be false in the 1980s (see \cite{Olsha}).  For semigroups, the Day-von Neumann conjecture does not make sense in general, since every semigroup is isomorphic to a subsemigroup of an amenable semigroup (see \cite{DayAmenable}). For example, a semigroup of rational maps  $S$ is right amenable whenever $S$ contains a constant endomorphism of $\overline{\C}$, say $g(z)=c$. Indeed the delta measure $\delta_g$ on the algebra of all subsets of $S$ defines a right invariant mean.  Moreover, $S$ is amenable whenever $c\in \overline{\C}$ is a common fixed point of all elements of $S$.
  
  If $S$ is a subgroup of M\"obius transformations (i.e. rational maps of degree $1$), then the answer to the Day-von Neumann conjecture can be deduced from Tits alternative theorem for linear groups (see \cite{TitsAlter}). Moreover, the semigroup  $S=\langle 2x, 2x +1 \rangle\subset \mathit{Aff}(\C)$ is free (see part (3) of Theorem \ref{th.MP}), but the affine group $\mathit{Aff}(\C)$ is amenable as it is the semidirect product of abelian groups.

  Nevertheless, as shown in \cite{CMAmenability}, for semigroups of non-constant rational maps containing a non-injective element $R$,  the condition of right amenability is connected with dynamical invariants of $R$ such as the measure of maximal entropy. More precisely, in \cite{CMAmenability} the following facts were proved:

\begin{enumerate}

\item \textit{If a finitely generated non-exceptional semigroup of rational maps $S$  admits a finitely additive measure, invariant under  the right action of $S$ onto itself, then $S$ has an invariant probability measure on $\overline{\C}$ which is the measure of maximal entropy for every non-injective element of $S$.}

\item \textit{Moreover, the right amenability of finitely generated semigroups of polynomials is equivalent 
to the right amenability of just one representation which we called the Lyubich representation (see definitions below).} 
\end{enumerate}

Later on in \cite{PakovichAmenRat} some of the results of \cite{CMAmenability} were generalized. See also \cite{Tucker}, \cite{ZieveZhan} and   \cite{PakovichSemiProb} and the discussion below. In \cite{CMAmenability}  a right amenable  version of the Day-von Neumann conjecture was divided into two parts regarding rational maps:

 The first problem related to the Day-von Neumann conjecture is formulated as: \textit{a semigroup $S$ of rational maps is right amenable if and only if $S$ admits an invariant probability measure $\mu$ which is the measure of maximal entropy of every non-injective element in $S$}.

This question was positively answered in \cite{CMAmenability} for finitely generated semigroups of non-exceptional polynomials. For semigroups of rational maps not conjugated to polynomials only the sufficiency was proved. In \cite{PakovichAmenRat}, the equivalence was extended to arbitrary semigroups of polynomials. The second part is formulated as: \textit{a semigroup $S$ of rational maps, not containing a rank $2$ free subsemigroup,  admits an invariant measure which is the measure of maximal entropy for every non-injective element of $S$.}

This question was positively answered in \cite{PakovichAmenRat} for semigroups of polynomials and semigroups of so-called tame rational maps. Later on, in \cite{Tucker}  it was  answered positively for arbitrary semigroups of non-injective rational maps (see Theorem \ref{th.Tuckermeas} below).

Recall that every non-injective rational map $R$ admits a unique ergodic non-atomic invariant probability measure $\mu_R$ with entropy $\log(deg(R))$ (see \cite{Brolin}, \cite{FreireLopesMane} and \cite{LyubichErgodicTheory}). 
The measure $\mu_R$ is known as the measure of maximal entropy.  Let $E(R)$ be the semigroup  of all rational maps $Q$ leaving $\mu_R$ invariant and such that $\mu_Q=\mu_R$ whenever $Q$ is non-injective. A polynomial (rational) map $Q$ is \textit{exceptional} if either $deg(Q)\leq 1$ or  $Q$ is affinely (M\"obius) conjugated to $z^n$ (with $n\in \mathbb{Z}$) or to a Chebyshev polynomial (or to a Latt\`es map). We say that a family of polynomial (rational) maps $\mathcal{F}$ is \textit{non-exceptional} whenever $\mathcal{F}$ contains a non-exceptional polynomial (rational) map. The following theorem and corollary complete the results of \cite{Tucker}, \cite{CMAmenability} and  \cite{PakovichAmenRat}. and answer conjecture 5.2 in \cite{PakovichSemiProb}.

\begin{theorem}\label{th.ERNeumann}
 Let $S<Rat$ be a  non-exceptional semigroup. Then the following statements are equivalent:
 \begin{enumerate}
  \item $S$ is right amenable.
  \item $S\subset E(R)$ for a suitable $R\in S$.
  \item $S$ does not contain a  rank $2$ free subsemigroup.
  \end{enumerate}
   
   Furthermore, if $S$ is finitely generated, the following condition is equivalent to the previous:\\
   
   4. $S$ has sub-exponential growth.\\
  
Finally, if $S$ is isomorphic to a non-exceptional polynomial semigroup, then the following is equivalent to conditions (1)-(3). \\

5. $S$ is contained in a finitely generated semigroup of subexponential growth.
 
\end{theorem}

\medskip

\textbf{Remark}. The equivalence (1) and (2) in the theorem above was shown in \cite{CMAmenability} for finitely generated semigroups of polynomials. For  semigroups of polynomials, the equivalence (1)-(3) of the theorem above was proved in \cite{PakovichAmenRat}. In \cite{PakovichAmenRat}, the implication from (3) to (1) was shown for semigroups of so-called tame rational maps. Also the implications from (3) to (2) and from (3) to (4) follows from results of \cite{Tucker}. Finally  
 from part (4) of Theorem \ref{th.ERNeumann} and theorems 1.1 and 1.2 in \cite{Tucker}, it follows that the properties of linear growth and subexponential growth coincide for finitely generated non-exceptional semigroups of non-injective rational maps.
\medskip

\textbf{Remark}. The equivalence of (1) and (3) can be regarded as a positive answer to the right amenable version of the Day-von Neumann conjecture for semigroups of rational maps. 

\medskip

For a non-injective polynomial $P$,  the semigroup $E(P)$ is finitely generated (see for example \cite{AtelaHu}). Comparing (2), (4) and (5) rises the question:

\begin{conjecture} Does there exist a rational map $R$ such that $E(R)$ is not finitely generated?
 
\end{conjecture}
 
\begin{corollary}\label{cor.charrigham} For a non-exceptional semigroup $S<Rat$, the following are equivalent:
\begin{enumerate}
 \item $S$ is right amenable.
 \item Every subsemigroup of $S$ is right amenable.
 \item Every rank $2$ subsemigroup of $S$ is right amenable.
 \item The Lyubich representation $\rho$  of $S$ is right amenable.
\end{enumerate} 
\end{corollary}

\medskip

\textbf{Remark}. As shown in Theorem \ref{th.ERNeumann}, the notion of right amenability is a natural extension of the notion of non-exponential growth from the finitely generated case to the general case of semigroups of rational maps. It would interesting to have a similar  characterization of right amenable semigroups of holomorphic endomorphisms of complex manifolds to Theorem \ref{th.ERNeumann}. For example, holomorphic endomorphisms of $\mathbb{CP}^k$ have a unique measure of maximal entropy. So, the implication from (1) to (2) in Theorem \ref{th.ERNeumann} carries on in this case with analogous arguments. 

On the other hand, for higher dimensional compact manifolds, there exist free abelian semigroups of holomorphic endomorphisms with uncountably many measures of maximal entropy, each of these measures is unrelated to the measures of other elements in the semigroup. For example, let $R_1$ and $R_2$ be rational maps, consider the semigroup $\langle P_1,P_2\rangle$ of endomorphisms of  $\overline{\C} \times \overline{\C}$, with $P_1(z,w)=(R_1(z),w)$ and $P_2(z,w)=(z,R_2(w))$.

The situation with semigroups of endomorphisms of non-compact manifolds is more complicated. For instance, already in complex dimension one, semigroups of entire functions are challenging since there is no notion of measure of maximal entropy for these maps. Moreover, in this situation is not clear  when the Julia sets of permutable entire functions coincide. This is a known open question for entire maps due to Baker \cite{BakerWan}.

Now we define a class of semigroups which we conjecture coincides with the class of right amenable semigroups of holomorphic endomorphisms of either complex compact manifolds or non-compact manifolds but of finite topological degree.

For a set $X$, let $End(X)$  be the semigroup of surjective endomorphisms and $F \subset End(X)$ be a finite family. We will say that 
a semigroup $S\subset End(X)$ belongs to the class $F$ whenever there exists a system of generators $\{f_1,...,f_n,..\}$ of $S$ 
such that for every  $g\in S$  there exist $\alpha \in F$, 
numbers $m\geq 0$,  $k_i\geq 0$, for $i=1,...,m$, and  a permutation $\sigma\in \mathcal{S}_m$ so that
$$g=\alpha\circ  f_{\sigma(1)}^{k_{\sigma(1)}}\circ ... \circ f_{\sigma(m)}^{k_{\sigma(m)}}$$ 
and the number $$\sum_{i=1}^m k_i$$ is comparable with the lenght of $g$ with respect to the generators $f_i$ by a constant independent of $g$. Here, the family $F$ not necessarily intersects $S$.

Given a semigroup $S$,  if there exist a set $X$ and  a faithful representation $\phi:S\rightarrow End(X)$ such that $\phi(S)$  belongs to the class ${F}$, for a suitable finite family ${F}\subset End(X)$, then  we say that $S$ belongs to the class $\mathcal{F}(X)$.

For example, a nearly abelian semigroup of rational maps belongs to the class $\mathcal{F}(\overline{\C})$. For the definition and further properties  of nearly abelian semigroup, see \cite{HinkMartSemi}. Also the semigroup $C(R)$, of rational maps commuting with a given rational map $R$, belongs to the class $\mathcal{F}(\overline{\C})$ (see \cite{PakovichFiniteness}).
If two rational maps $Q, R$ satisfy the Levin relations (that is, $R\circ Q=R^2$ and $Q\circ R=Q^2$), then $\langle Q, R \rangle$ belongs to the class $\mathcal{F}(\overline{\C})$.

\begin{proposition}\label{Pr.RX} A semigroup $S$ is right amenable whenever $S$ belongs to the class $\mathcal{F}(X)$ for some set $X$.
\end{proposition}

\textbf{Remark}. Theorem \ref{th.ERNeumann} and theorem 4.2 in \cite{HinkMartSemi} implies that every right amenable semigroup $S$ of polynomials is nearly abelian and hence belongs to the class $\mathcal{F}(\C)$. Thus, the converse statement to Proposition \ref{Pr.RX} holds true for semigroups of polynomials.

Also, Theorem \ref{th.ERNeumann} shows that a right amenable semigroup of non-constant rational maps either does not contain exceptional maps or consists only of exceptional maps. But, as the example of monomial semigroups shows,  
the situation with semigroups of exceptional maps is  different.   

We denote by $MP$ the semigroup of monomial polynomials with composition. Hence $MP$ is generated by the semigroup $Z=\langle z^n\rangle_{n\in \mathbb{N}}$  and the homotety group $H=\{\lambda z: \lambda\in \mathbb{C}^*=\C\setminus 0\}$.

\begin{theorem}\label{th.MP} The following statements hold true. 

\begin{enumerate}
 \item $MP$ is amenable.
 \item For every $P\in MP$ with $deg(P)\geq 2$, the semigroups $E(P)$ and $E(P)\cap MP$ are amenable and  conjugated to $E(z^2)$ and $E(z^2)\cap MP$, respectively. 
 \item  For every $n\geq 2$, the semigroup
$E(z^n)=E(z^2)$ and contains free subsemigroups of any given rank. 
\end{enumerate}
Moreover,
\begin{enumerate}
\item[4.] Let  $U\subset H$ and $T\subset Z$ be semigroups. Then, the semigroup $S=\langle U,T \rangle \subset MP$ is right amenable and contains a free subsemigroup of rank $2$, whenever $U$ contains an element of infinite order.
  \end{enumerate}
\end{theorem}
 
Left amenability of semigroups of rational maps is more delicate than right amenability as we see next. 

A semigroup $S$ is called \textit{left cancellative} if for every $a,b,c \in S$  the equation 
$ab=ac$ implies $b=c$. Right cancellative semigroups are defined analogously. A semigroup $S$ is called \textit{cancellative} if it is both left and right cancellative. For example, any subsemigroup of a group is cancellative. A semigroup of surjective endomorphisms of a set $X$ is always right cancellative. 

Recall that $RIM(S)$ ($LIM(S)$) denotes the set of right (left) invariant means of the semigroup $S$. 

\begin{theorem}\label{th.RIMLIM}
Let $S$ be a  non-exceptional semigroup of rational maps,  then we have:
\begin{itemize}
\item If $RIM(S)\subset LIM(S)$, then $S$ is cancellative and does not contain free subsemigroups of rank $2$.
\item On the other hand, if $S$ is cancellative without free subsemigroups of rank two, then $S$ is amenable. 
 \end{itemize}
\end{theorem}

The following is an immediate corollary of Theorem  \ref{th.ERNeumann} and Theorem \ref{th.RIMLIM}.

\begin{corollary}
A right amenable semigroup of rational maps $S$ is amenable whenever is left cancellative. 
\end{corollary}

For semigroups of polynomials, we have a stronger result: 

\begin{theorem}\label{th.three} Let $S$ be a non-exceptional left amenable semigroup of polynomials, then $LIM(S)=RIM(S)$.
\end{theorem}

The case for rational maps is still unknown, so the following question is natural:

\begin{conjecture}\label{conj.uno} Let $S$ be a non-exceptional left amenable semigroup of rational maps. Does the set of left invariant means coincide with the set of right invariant means?
\end{conjecture}

\textbf{Remark}. The theorem and the corollary above are connected with Sorensen's conjecture  which states: \textit{a right cancellative left amenable semigroup is left cancellative.} This conjecture was proved by Sorensen in his thesis for a class of so-called left measurable semigroups (left measurability condition is stronger than left amenability, see \cite{Klawe}). But in \cite{Klawe} Sorensen's conjecture was disproved in the non-measurable case.

Since every semigroup of rational maps is always right cancellative, then the first part of Theorem \ref{th.RIMLIM} confirms the Sorensen's conjecture for rational maps whenever $RIM(S)\subset LIM(S)$. Hence, for polynomials, Theorem \ref{th.RIMLIM} and Theorem \ref{th.three} immediately imply:

\begin{corollary}\label{cor.amenable}
A countable non-exceptional semigroup of polynomials $S$ is amenable if and only if the following two conditions are satisfied.
        \begin{itemize}
      \item $S$ does not contain  free  subsemigroups of rank $2$.                                                             \item $S$ is cancellative.
                                                                        \end{itemize}

        \end{corollary}

The cancellativity condition in the previous corollary leads to the important notion of embeddability of semigroups. We discuss this notion in the next subsection. 

\subsection{Suskievich theorem}

Motivated by a question of van der Waerden, in \cite{Sushkevich}, Sushkievich published a theorem which states that a cancellative finitely generated semigroup is always embeddable 
into a group. Later on, in \cite{Malcev},  Maltsev found a counterexample to Suskievich's statement and gave necessary and sufficient conditions for a semigroup to be embeddable into a group. As mentioned in \cite{Hollings}, Maltsev condition consists of infinitely many equations which, in practice, are difficult to check.   In \cite{OreSem}, O. Ore gave sufficient conditions of embeddability of a semigroup into a group:

\begin{theorem}[Ore Theorem]\label{th.Ore} A cancellative semigroup $S$ is embeddable into a group whenever every pair of principal left (right) ideals have non-empty intersection.
\end{theorem}

The intersection property in   Theorem \ref{th.Ore} which is  known as right (left)  reversibility neither easy to check. For rational maps, this property was related with an orbit intersection property (see \cite{CMAmenability}).  Every left (right) amenable semigroup $S$ is left (right) reversible, moreover for a class of rational maps the notion of right (left) reversibility is equivalent to the notion of left (right) amenability (see \cite{PakovichAmenRat}). Hence by Theorem \ref{th.ERNeumann} and Theorem \ref{th.Ore} a semigroup of rational maps, which does not contain free subsemigroups of rank $2$, is embeddable into a group whenever it is cancellative. The latter statement together with Corollary \ref{cor.amenable}, immediately imply the following.

\begin{corollary}\label{Cor.leftAmaneble}
 A left amenable semigroup of polynomials is always embeddable into a group. 
\end{corollary}

Summarizing the discussion above, it seems that the cancellativy condition is enough for rational maps. Thus, Sushkievich statement is plausible in this case as a conjecture:

\begin{conjecture} Is it true that a  semigroup of rational maps is embeddable into a group whenever it is cancellative?
 
\end{conjecture}

The following theorem  together with results from \cite{CMAmenability} motivate the interest in the conjecture above.

\begin{theorem}\label{th.ptaktheorem}
Given a non-exceptional polynomial $P$, the following  statements hold true.  
\begin{enumerate}
\item The semigroup $E(P)$ is amenable if and only if $E(P)$ is embeddable into the group of conformal automorphisms of a suitable Riemann surface. 

\item Moreover, $E(P)$ is abelian if and only if $E(P)$ is embeddable into the group of conformal autormorphisms  of the Riemann sphere. 
\end{enumerate}

\end{theorem}
By Theorem \ref{th.Ore}, a left amenable semigroup of rational maps is embeddable into a group if and only if satisfies Sorensen's conjecture.

In support of Sushkievich's conjecture, we have the following results.

\begin{theorem}\label{th.embeddable}
 A cancellative semigroup of rational maps $S$ is embeddable into a group whenever one of the following conditions holds.
 \begin{enumerate}

    \item  $S$ is isomorphic to a subsemigroup of $MP$.
  
  \item $S\simeq \langle R,Q\rangle$ where $R,Q$ are non injective rational maps sharing a common superattracting fixed point.

  \item The semigroup $S$ is isomorphic to  the free product of right amenable subsemigroups.
  
 \end{enumerate}

\end{theorem}

Finally, according to the theorems above, we propose the following general conjecture.

\begin{conjecture}
 Does every finitely generated semigroup of rational maps contain a subsemigroup of finite index which is a free product of  its right amenable subsemigroups? 
\end{conjecture}

For cancellative semigroups the following version seems to hold true.
\begin{conjecture}\label{conj.free}
Is it true that every finitely generated cancellative semigroup of rational maps is a free product of its amenable subsemigroups?
\end{conjecture}

Last conjectures  may hold true for countably generated semigroups of rational maps but, as noticed by the referee, it could rather  be an amalgameted free product. 

These conjectures remind  the Klein-Maskit combination theorems in the setting of semigroups of rational maps.

\subsection{The connection with Furstenberg $\times 2 \times 3$ problem.}

To establish the connection of the ideas and discussion above with Furstenberg's problem, we start with the following theorem. 

\begin{theorem}\label{th.leftKoopman}
 A semigroup of rational maps with left amenable Koopman (anti)-representation admits a common invariant probability measure on $\overline{\C}$. 
\end{theorem}

As an immediate corollary, we have:
\begin{corollary}\label{cor.leftamena}
 A left amenable semigroup of rational maps admits a common invariant probability measure on $\overline{\C}$. 
\end{corollary}

According to \cite{CMAmenability}, the semigroup $\langle z^2+c,-z^2-c \rangle\subset E(z^2+c)$ is not left amenable.  Hence, the reciprocal of Corollary \ref{cor.leftamena} is  not true. Moreover, any non-left amenable semigroup of polynomials (for example, a free semigroup of polynomials) admits an invariant measure on $\overline{\C}$, namely the delta measure $\delta_\infty$ based at infinity. In other words, a common invariant measure $\mu$ for rational maps $Q$ and $R$ may not generate a left invariant mean for the semigroup $\langle Q, R \rangle$. But, by Theorem \ref{th.ERNeumann}, if $\mu$ is the measure of maximal entropy for $Q$ and $R$, then $\mu$ necessarily generates a right invariant mean for $\langle Q, R\rangle$ and, moreover by Theorem \ref{th.three} it generates a left invariant mean whenever $Q$ and $R$ are non-exceptional polynomials generating a left amenable semigroup.

Besides, we can not identify the invariant measure given by Corollary \ref{cor.leftamena} for a left invariant mean. By theorem 3 in \cite{CMAmenability}, every right invariant mean generates the same measure of maximal entropy.

 In general, a right amenable semigroup of polynomials may have non-atomic invariant measures apart from the measure of maximal entropy. For example, as shown in theorem 28 of  \cite{CMAmenability}, the semigroup $\langle -z^2-\lambda z^4, z^2+\lambda z^4 \rangle$ is right amenable for every $\lambda\in \C$. For sufficiently small $\lambda$, the polynomials $z^2+\lambda z^4$ and $-z^2-\lambda z^4 $ are hyperbolic and  share an invariant topological disc $D$ which is the closure of the immediate component of attraction of $0$. Thus, these polynomials share the same invariant harmonic measure on $\partial D$ which is the pull-back of the one-dimensional Lebesgue measure on the unit circle with respect to the B\"ottcher coordinate.  Note that $z^2+\lambda z^4$ and $-z^2-\lambda z^4 $ do not share a common iteration but share the  common right factor $z^2$ and, the one dimensional Lebesgue measure on the unit circle is the measure of maximal entropy for $z^2$.

From the disscusion above we arrive to an old general dynamical question which we formulate in our setting as follows:

\textit{Given two non-injective rational maps $Q$ and $R$;  how big  $M(Q)\cap M(R)$ should be in order to induce a relation in the semigroup $\langle Q, R \rangle$?}  

Here $M(R)$ denotes the space of $R$-invariant probability measures. If $Q$ and $R$ are two rational maps sharing the measure of maximal entropy then by Theorem \ref{th.ERNeumann} the semigroup  $\langle Q, R \rangle$ is not free (see \cite{LevinPrytyckiRel} and \cite{Ye}). On the other hand, if $Q$ and $R$ do not share the measure of maximal entropy, then $\langle Q , R  \rangle$ contains a free subsemigroup of rank $2$ (see \cite{Tucker} and Theorem \ref{th.ERNeumann}).

For free semigroups we show in next proposition that there exists a  rank $2$ semigroup $S$ of rational maps not admiting a finite invariant measure. 
On the contrary, the semigroup $S=\langle \lambda z^2, z^2\rangle$ with $\lambda=2\pi i\alpha$ and $\alpha\in \mathbb{R}\setminus \mathbb{Q}$ is free but the one-dimensional Lebesgue measure on the unit circle is $S$-invariant (see part (2) of Theorem \ref{th.MP} or Theorem \ref{th.ZieveZhan}, below). 

\begin{proposition}\label{pr.semigroupwithout}
 There exists a rank $2$ free semigroup of rational maps $S=\langle Q,R \rangle$  such that $Q$ and $R$ do not share a common invariant measure. In particular, $S$ is free and, more, is not Koopman left amenable.
\end{proposition}

 Proposition \ref{pr.semigroupwithout} is an  application of theorem 2.3 of  \cite{HinkMartSemi}. In that paper, the authors also adapted the arguments of the proof of Klein's combination theorem in order to produce a geometrical condition under which a semigroup of rational maps is a free product of proper subsemigroups. According to Proposition \ref{pr.semigroupwithout} 
such semigroups do not admit  invariant measures regardless of their subsemigroups.

 We call a measure $\mu$ a \textit{discrete measure} when $\mu$ is a finite linear combination of delta measures $\delta_a$ based at the points $a\in \overline{\C}$. A discrete measure is invariant if and only if it is supported on a periodic cycle.

The following proposition is an application of theorem 1.2 in  \cite{BakerdeMarco} and theorem 1.5 in \cite{yuanzhang}.
 
\begin{proposition}\label{pr.LdM} Let  $Q$ and $R$ be rational maps. If $M(Q)\cap M(R)$ contains infinitely many non linearly dependent discrete measures, then $Q$ and $R$ share  the same measure of maximal entropy.
\end{proposition}

Note that if $M(Q)\cap M(R)$ contains all discrete invariant measures based on the fixed points of $R^2$, for maps $R$ and $Q$ with $deg(Q)=deg(R)$, then $Q=R$ whenever $R^2$ has the maximal number of fixed points. This is because  $Q$ and $R$ coincide on each fixed point of $R^2$, which induces more equations than parameters.

Summarizing all the above, we propose the following conjecture.

\begin{conjecture}\label{conjetura1} Assume two rational maps $Q$ and $R$ admit a common ergodic invariant probability measure $\mu$. Is it true that either $\mu$ is discrete or the maps $Q$ and $R$ share the measure of maximal entropy? 
 
\end{conjecture}

Since the support of an invariant measure  is an invariant compact subset of the Riemann sphere, this  conjecture may be stated in terms of invariant compact sets.

Also note that Open question \ref{conjetura1} is related to the   Furstenberg $\times 2$  $\times 3$ problem, which we reformulate as follows:

\textit{For the semigroup $\langle z^2, z^3 \rangle$, the only non-atomic ergodic invariant probability measure on $\overline{\C}$ is the one-dimensional Lebesgue measure  on the unit circle.}

For further information and develoments on Furstenberg's problem, see  \cite{Furstenberg},  \cite{JohnsonFurs} and \cite{RudolphFurs}.

The one-dimensional Lebesgue measure on the unit circle is the measure of maximal entropy for every element of $\langle z^2, z^3\rangle$. The semigroup 
$S=\langle z^2, z^3 \rangle$ is  exceptional and abelian and, hence, amenable with $RIM(S)=LIM(S)$. Also $z^2$ and $z^3$ are primitive polynomials (see definition below).

To support the previous conjecture, we prove the following theorem.

\begin{theorem}\label{th.Theorem5}
 Let $S$ be a semigroup of  non-injective polynomial maps admiting  an invariant probability measure $\mu$. If the support $supp(\mu)$ belongs to the Fatou set $F(P)$, for $P\in S$, then either $\mu$ is discrete  or every element of $S$ shares a common iterate with $P$. In particular, $S$ is amenable. 
\end{theorem}

Thus, according to the theorem above, the support of any finite non-atomic  invariant measure for a non-amenable semigroup $S$ must be contained in the intersection of the Julia sets of all the elements of $S$.  

A rich source of  examples of  invariant probability measures for a rational map $R$, appart from the measure of maximal entropy, is given when the Julia set of $R$ contains a homeomorphic copy of the Julia set of another rational map.  For instance, a harmonic measure supported on the boundary of a common periodic component.

\begin{definition}
A rational map $R$ admits a \textit{marked  polynomial-like measure} $(\mu,z_0)$  whenever:
\begin{itemize}

\item There exists a pair of topological discs $D_0$ and $D_1$ such that $D_0\Subset D_1$ and  the restriction of $R$ on $D_0$ defines a polynomial-like map $T:D_0\rightarrow D_1$  with connected Julia set $J(T)$  and $\mu$ on $D_0$ is the respective copy of  the measure of maximal entropy for the unique polynomial in the hybrid class of $T$. 

\item The point $z_0\in J(T)$ is periodic.
\end{itemize}
\end{definition}

For definitions and further information on polynomial-like maps see for example \cite{DouadyHubbardPolike}.

\begin{theorem}\label{th.polynomiallike}
 Let  $Q$ and $R$ be rational maps so that $Q^n$ and $R^m$ admit the same marked polynomial-like measure for suitable $m,n\geq 1$, then $Q$ and $R$ generate a right amenable semigroup. Moreover, if either $Q$ or $R$ is non-exceptional, then $\langle Q, R \rangle$ is amenable. 
\end{theorem}

According to theorem 30 in \cite{CMAmenability} every pair of non-exceptional rational maps generating an  amenable semigroup $S$ with $RIM(S)=LIM(S)$ share a common iteration.  Furthermore, this pair may share infinitely many common non-atomic invariant measures.

Therefore  the following stronger version of Open question \ref{conjetura1}  seems a natural generalization of Furstenberg's $\times 2 \times 3$ problem for non-exceptional semigroups of rational maps.
 
\begin{definition} 
A non injective rational map $R$ is called  \textit{prime} if whenever $R=G\circ T$, then either $G$ or $T$ is an injective map.
 \end{definition}

\begin{conjecture}\label{conj.Furstenration}
 Let $Q$ and $R$ be two non injective prime rational maps without common iteration and such that $Q$ and $R$ admit a common invariant probability measure $\mu$. Is it true that $\mu$ is either discrete or $\mu$ is the measure of maximal entropy of every element of the semigroup $\langle Q, R \rangle$?
 \end{conjecture}

Next theorem answers our generalization of Furstenberg $\times 2 \times 3$ problem  for a right amenable semigroup of non exceptional polynomials without pairwise common iteration.

\begin{definition}
 Let $Deck(R)=\{\gamma\in Mob | R\circ \gamma=R\}$ be the deck group of the rational map $R$. We call a map $R$ \textit{primitive} if $Deck(R)$ is a cyclic group of order $deg(R)$.
\end{definition}

\begin{theorem}\label{th.free}
Let $S$ be a right-amenable  semigroup of polynomials.  Let $P,Q\in S$ be two non-injective polynomials not sharing a common iteration and such that $P$ is primitive and non-exceptional. Assume that $S$ admits an invariant probability measure $\mu$, then  
 $\mu$ is either discrete or  the measure of maximal entropy of $P$.
\end{theorem}

\textbf{Acknowledgement.} The authors would like to thank the referee for useful comments and careful reading of a previous version of this paper. 

\section{Some background in holomorphic dynamics and amenability of semigroups}

Let us recall some basic concepts on holomorphic dynamics and amenability of semigroups. 
\subsection{Amenability of semigroups}

Every semigroup $S$ canonically defines a right and left action of $S$ onto itself. A
semigroup is called \textit{right (left) amenable} if the right (left) action of $S$ onto itself admits 
an invariant, finitely additive, probability measure $\mu$. This measure defines a \textit{mean}, which is a normalized positive right (left) invariant continuous functional on $L_\infty(S)$.  If $RIM(S)\cap LIM(S) \neq \emptyset$,  $S$ is called  an \textit{amenable semigroup}.

Recall some basic properties of amenable semigroups, see for example \cite{DayAmenable}.

\begin{itemize}
 \item Every abelian semigroup is amenable.
 \item Every semigroup is a subsemigroup of an amenable semigroup.
 \item Every subgroup of an amenable group is amenable. 
\end{itemize}

In this article, we use a weaker property than amenability, namely $\rho$-amenabi-lity based on  \cite{DayMeans}.

First, we say that a proper right (left) $S$-invariant  subspace $X\subset L_\infty(S)$ is called either \textit{right} or \textit{left amenable} if $X$ contains constant functions and there exists a mean $M$ so that the restriction of $M$ to $X$ is an invariant functional for either  the right or left actions of $S$ on $X$, respectively. In other words, $X$ is right (left) $S$-invariant and admits a respective invariant non-negative state. Note that every semigroup admits an amenable subspace, for example,   the subspace  of constant functions is always amenable. 

Now, let $\rho$ be a bounded (anti-)homomorphism from $S$ into $End(B)$, where $B$ is a Banach space, and $End(B)$ is the semigroup of continuous linear endomorphisms of $B$. Let $B^*$ denote the dual space of $B$. Given a pair 
$(b,b^*)\in B\times B^*$ consider the function $f_{(b,b^*)}\in L_\infty(S)$ given by
$$f_{(b,b^*)}(s)=b^*(\rho(s)(b))).$$  Let $Y_\rho\subset L_\infty(S)$, be the closure of the linear span of the family of functions $\{ f_{(b, b^*)} \}$ for all pairs $(b,b^*)\in B\times B^*.$ Finally let $X_\rho$ be the space generated by $Y_\rho$ and the constant functions. Note that $X_\rho$ and $Y_\rho$ are both right and left invariant. 

\begin{definition}
We will say that a bounded (anti-)homomorphism $\rho$ is either \textit{right amenable or left amenable} whenever $X_\rho$ is either a right or left amenable subspace of $L_\infty(S)$, respectively.  Also we will say that $S$ is $\rho$ right amenable or $\rho$ left amenable whenever $\rho$ has the respective property. Equivalently, that the $\rho$-action of $S$ on $B$ is either right amenable or left amenable, respectively.
\end{definition}

To show that, in general, amenability is stronger than $\rho$-amenability let us recall Day's theorem from \cite{DayMeans}.

\begin{theorem}\label{th.Day} A  semigroup $S$ is right amenable (left amenable) if and only if $S$ is $\rho$-right amenable ($\rho$-left amenable) for every bounded (anti)representation $\rho$.
 
\end{theorem}

A semigroup $S$ satisfies the right (left) F\o{}lner condition (FC) if for every finite set $H$ and every $\epsilon>0$ there exists a finite set $F$ such that $card(Fs\setminus F)<\epsilon\,card(F)$ for every $s\in H$ (respectively, $card(sF\setminus F)<\epsilon\, card(F)$ for every $s\in H$).

A semigroup $S$ satisfies the right (left) strong F\o{}lner condition (SFC) if for every finite set $H$ and every $\epsilon>0$ there exists a finite set $F$ such that $card(F\setminus Fs)<\epsilon\,card(F)$ for every $s\in H$ (respectively, $card(F\setminus sF)<\epsilon\, card(F)$ for every $s\in H$).

The following theorem summarizes Theorem 2.2 in \cite{Klawe} (see also \cite{Gray}).

\begin{theorem}\label{th.Folner} For every semigroup $S$,
$$\textnormal{ left SFC } \Rightarrow \textnormal{ left amenability } \Rightarrow \textnormal{ left FC}.$$

None of the implications above are reversible. However, these implications are reversible for left cancellative semigroups. 
 
\end{theorem}

Let $U$ and $T$ be  semigroups with a homomorphism $\varphi:T\rightarrow End(U).$
Let  $U\rtimes_\varphi T$ be the semidirect product with multiplication
$$(u_1,t_1)\cdot (u_2,t_2)=(u_1\cdot \varphi(t_1)(u_2),t_1t_2).$$ 

Next theorem gathers two results of M. Klawe (see \cite{Klawe}).

\begin{theorem}\label{th.Klawe} Let $U$ and $T$ be semigroups and $\varphi:T\rightarrow End(U)$ be a homomorphism.
\begin{itemize}
 \item If $U$ and $T$ are right amenable, then $U\rtimes_\varphi T$ is right amenable.
\item If $U$ and $T$ are amenable semigroups and  $\varphi(t)$ is a surjective map for every $t\in T$, then $U\rtimes_\varphi T$ is amenable.
\end{itemize}

\end{theorem}

\begin{proof}
These appear as  Proposition 3.10  and Corollary 3.11 in \cite{Klawe}.
\end{proof}

\subsection{Dynamics of semigroups of rational maps}

We assume that the reader is familiar with basic notions of holomorphic dynamics, such as the Julia set, the Fatou set (see for example \cite{Mdyn}) and polynomial-like maps (see \cite{DouadyHubbardPolike}).

An important notion in holomorphic dynamics is that of the measure of maximal entropy. According to Brolin \cite{Brolin} (polynomial case),  Lyubich \cite{LyubichErgodicTheory} and, independently, by Friere, Lopez and Ma\~n\'e \cite{FreireLopesMane}  (rational case) every non injective rational map $R$ has a unique invariant probability measure $\mu_R$ of maximal entropy supported in the whole Julia set $J(R)$.

Due to results in \cite{LevinPrytyckiRel}, \cite{Levinrelations}  and \cite{Ye},  we have the following  characterization under which two non-exceptional rational maps share the same measure of maximal entropy. 

\begin{theorem}\label{th.Levinexc}
Two non-exceptional rational maps $Q$ and $R$ share the same measure of maximal entropy if and only if there are iterations of  $Q$ and $R$ which satisfy the relations $$R^n\circ Q^m= R^{2n}$$ and $$Q^m\circ R^n=Q^{2m}.$$

\end{theorem}

Moreover, in connection with the algebraic structure of semigroups we have the following theorems showed in \cite{HinkMartSemi},  in \cite{ZieveZhan}, and  in \cite{Tucker}.

Let $S$ be semigroup of rational maps and $U$ be non-empty open subset of $\overline{\C}$, then $U$ is called a backward fundamental set for $S$ whenever $g^{-1}(U)\cap U=\emptyset$ for $g\in S\setminus {Id}.$ The following theorem is Theorem 2.3 in \cite{HinkMartSemi}.
\begin{theorem}\label{th.HM} Let $U$ and $V$ be two backward fundamental sets for semigroups of rational maps $G$ and $S$, respectively.
Then the semigroup $\Gamma=\langle G, S\rangle$ is freely generated by $G$ and $S$ whenever $U,V$ satisfy the Klein combination condition, namely:
$$\overline{\C}\setminus U\subset V$$
and
$$\overline{\C}\setminus V\subset U.$$
\end{theorem}

The following theorem was proved in \cite{ZieveZhan}.
\begin{theorem}\label{th.ZieveZhan}
Let $Q$ and $R$ be rational maps sharing a common superattracting fixed point, then the following statements are equivalent:
\begin{enumerate}
 \item The semigroup $\langle Q,R \rangle$ is not free.
 \item There are integers $k,n,m\geq 1$ such that $R^n\circ Q^m$ commutes with $R^k$.
\end{enumerate}

\end{theorem}

The following theorem was proved in \cite{Tucker}.
\begin{theorem}\label{th.Tuckermeas}
 Let $Q$ and $R$ be rational maps of degree at least $2$.
 If $m_Q\neq m_R$,  then there exist a $n$ so that the semigroup $\langle Q^n, R^n \rangle$ is 
 free.
\end{theorem}

We believe the following question about the algebraic structure of rational semigroups has afirmative answer.

\begin{conjecture} Let $S$ be a finitely generated semigroup of non-injective rational maps which contains a free semigroup in two symbols $A,B\in S$. Is it true that there exists  a finite index subsemigroup $S_0<S$ where $S_0=S_1\ast S_2$ is a free product with $A\in S_1$ and $B\in S_2$?
\end{conjecture}

\subsection{Koopman representation}

Every rational map $R$ acts on the space $\mathcal{C}(\overline{\C})$ of continuous functions by the formula 

$$K_R(\phi)= \phi(R),$$ $K_R$ is known as Koopman operator.
The correspondence $K:Rat\rightarrow End(\mathcal{C}(\overline{\C}))$ given by $R\mapsto K_R$ is a bounded anti-homomorphism. 

Let us note that Koopman operator $K_R$ is a right inverse to Lyubich operator. For a rational map $R$, the Lyubich operator $L_R$ is the endomorphism of $\mathcal{C}(\C)$  (see \cite{LyubichEntropy} and \cite{CMAmenability}) given by the formula.

$$L_R(\phi)(x)=\frac{1}{deg(R)}\sum_{y\in R^{-1}(x)} \phi(y).$$
Hence the measure of maximal entropy $\mu_R$  is a fixed point for the dual operator $L_R^*$ acting on the space of complex valued finite measures which is a dual space to $\mathcal{C}(\C)$.
Therefore the semigroup $E(R)$ also can be defined as  $$E(R)=\{Q\in Rat(\C): L^*_Q(\mu_R)=\mu_R\}.$$  

\subsection{Embeddability of semigroups}
To finish this section we return to the discussion on the embeddability of a semigroup into a group. That is to specify under which circumstances a given semigroup $S$ is ``half'' of a group.

Let $\Gamma$ be a countable  group with a minimal set of generators $\langle \gamma_1,...,\gamma_n,...\rangle$, consider the subset $\Gamma_+$ of all words in the alphabet $\{\gamma_1,...\gamma_n,...\}$. Then $\Gamma_+$ forms a countable semigroup which is called the \textit{positive part} of $\Gamma.$ Note that 
$\Gamma$ is generated by $\Gamma_+$ and $(\Gamma_+)^{-1}$. A countable semigroup $S$ is \textit{embeddable into a group} if $S$ is isomorphic to the positive part of a group.

As mentioned above, Theorem \ref{th.Ore} gives a sufficient condition for a semigroup to be embeddable into a group.

After the results of Maltsev appeared, Lambek and Pt\'ak gave another characterization of embeddable semigroups. 
We will need Pt\'ak's conditions for our proves which are summarized next.

Let $S$ be a semigroup generated by a set of generators $T$, consider $T$ as an alphabet, let $FS_T$ be a free semigroup given by the words on the alphabet $T$ and let $FG_T$ be the free group on the alphabet $T$. Clearly $FS_T$ is a subsemigroup of $FG_T$, hence $FS_T$ is embeddable. Let $\pi_S:FS_T\rightarrow S$ be the canonical projection, then the first result of Pt\'ak is that $\pi_S$ is a well-defined surjective homomorphism. So, $\pi_S$ defines a congruence $\sim_S$ on $FS_T$ by $a\sim_S b$ if and only if $\pi_S(a)=\pi_S(b)$. In other words, $$FS_T/\sim_S\simeq S.$$  Now we can formulate the main result of Pt\'ak (see, for example, \cite{Hollings}).

\begin{theorem}\label{th.Ptak}
Let $S$ be a semigroup with a minimal system of generators $T$, then $S$ is embeddable into a group if and only if the congruence $\sim_S$ extends from $FS_T$ to a congruence on $FG_T$ which is defined by a normal subgroup $N<FG_T$, so that $S$ is embeddable into $FG_T/N$.
\end{theorem}

\section{Proofs of theorems}

Now we are ready to prove our theorems.

The proof of  Theorem \ref{th.ERNeumann} is based on results from \cite{Tucker}, \cite{CMAmenability} and \cite{PakovichAmenRat}.
\begin{proof}[Proof of Theorem \ref{th.ERNeumann}]
First we show the equivalences (1)-(3).
(1) implies (2).  For semigroups of polynomials follows from theorem 1.2 in \cite{PakovichAmenRat}. For semigroups that are not conjugated to semigroups of polynomials, follows from theorem 19 in \cite{CMAmenability}. 

(2) implies (3).  Let $S=S_0\sqcup S_1$ be the disjoint union of the semigroups $S_0$ and $S_1$ of injective and non-injective elements in $S$, respectively. 
By lemma 2 of \cite{LevinPrytyckiRel}  $S_0$ is a finite semigroup and hence  $S_0$ is a group. Then every rank $2$ free subsemigroup must be contained in $S_1$, but this is forbidden by Theorem \ref{th.Levinexc}. 

(3) implies (1). By proposition 3.4 in \cite{paradoxal}, it is enough to consider the case when $S$ is countable.  As $S$ is countable, let  $\Gamma=\{\gamma_1,\gamma_2,... \}$ be a countable set of generators.  If $S_n=\langle \gamma_1,...,\gamma_n \rangle$, then $S=\cup S_n$ is an exhaustion by finitely generated semigroups. Recall  Day theorem  \cite{DayAmenable} which states:

\textit{Let $S=\bigcup S_n$,  where $S_n$ are  semigroups such that for every $m,n$ there exists $k$ with $S_m \cup S_n \subset S_k$. Then $S$ is right (left) amenable whenever the semigroups $S_n$ are right (left) amenable for every $n$.}

Now,  we need to show that $S_n$ is right amenable for every $n$.  Indeed, by assumptions and by the theorem 1.2 in \cite{Tucker}, the semigroups $S_n$ has linear growth and, by theorem 4.3 in \cite{Gray}, $S_n$ satisfies both the right and left F\o{}lner conditions. Since $S_n$ is right cancellative then the opposite semigroup $\tilde{S}_n$ is left-cancellative and satisfies the left F\o{}lner condition, hence $\tilde{S}_n$ is left amenable by Theorem \ref{th.Folner}. Therefore $S_n$ is right amenable.

If $S$ is finitely generated, then the implication   (3) to (4) follows from theorem 1.2 in \cite{Tucker}. Reciprocally, to prove (4) to (3) consider $S$ a finitely generated semigroup of subsexponential growth, by theorem 4.3 in \cite{Gray}, $S$ satisfies both left and right F\o{}lner conditions, then analogous  arguments as in the proof from (3) to (1) apply.

In the polynomial case, the equivalence (2)-(5) follows from  part (4) and the fact that $E(P)$ is always finitely generated for a non-exceptional polynomial $P$ (see, for example \cite{AtelaHu}).

\end{proof}

\begin{proof}[Proof of Corollary \ref{cor.charrigham}]
 (1) implies (2) by item (2) of Theorem \ref{th.ERNeumann}. Clearly (2) to (3). (3) implies (1) again by part (2) of Theorem \ref{th.ERNeumann}.
 Since Lyubich representation is bounded then (1) implies (4) by Theorem \ref{th.Day}. Now, (4) implies (2) by theorem 19 in \cite{CMAmenability}.
 
\end{proof}
\begin{proof}[Proof of Proposition \ref{Pr.RX}]
It is enough to prove that every semigroup $S\subset End(X)$ in the class $F$ for  given set $X$ and finite family $F\subset End(X)$ is right amenable. 
First assume that $F=\{Id\}$, then $$S=\bigcup_{m=1}^\infty S_m$$ where $$S_m=\langle f_1,...,f_m\rangle.$$
By the condition, every $g\in S_m$ has the representation $$g= f_{\sigma(1)}^{k_{\sigma(1)}}\circ ... \circ f_{\sigma(m)}^{k_{\sigma(m)}}$$
for suitable natural numbers $k_i\geq 0$ and a permutation $\sigma$ so that the number 
$$\sum_{i=1}^m k_i$$ is uniformly comparable with the length of $g$ in $S_m$. Hence
$$card(g\in S_m: length(g)\leq n)$$ is bounded by a number comparable with  $m!(\frac{n(n+1)}{2})^m$. Thus, $S_m$ is a semigroup of polynomial growth, so it satisfies both right and left F\o{}lner conditions by theorem 4.3 in \cite{Gray}. Since $S_m$ is right cancellative, then $S_m$ is right amenable by Theorem \ref{th.Folner}. Hence $S$ is right amenable as $S$ is a countable union of nested right amenable semigroups. 

When $F$ is a finite family, we only have to note that $$card(g\in S_m: lenght(g)\leq n)$$ is bounded by a number comparable with  $card(F) m!(\frac{n(n+1)}{2})^m$, and the same arguments apply to finish the proof.
 
\end{proof}

\begin{proof}[Proof of Theorem \ref{th.MP}]

Part 1.    Let  $H=\{\lambda z: \lambda\in \mathbb{C}^*\}$ and $Z=\langle z^n \rangle_{n\in \mathbb{N}}$. Since $z^n$ acts on $H$ by semiconjugacy as a surjective endomorphism for every $n\geq 1$, there exists a representation $\rho:Z\rightarrow End(H)$ so  that $$MP\simeq H\rtimes_\rho Z$$ and then apply Theorem \ref{th.Klawe} to finish this part.

Part 2. Every $P\in MP$ with $deg(P)\geq 2$ is M\"obius conjugated 
to $z^{deg(P)}$, the semigroup $E(P)$ is conjugated to $E(z^{deg(P)})$ and, in the same way, $E(P)\cap MP$ is conjugated to $E(z^{deg(P)})\cap MP$.
As the measure of maximal entropy of $z^n$ is the one-dimensional Lebesgue measure on the unit circle for all $n>1$, then  $Mob\cap E(z^{deg(P)})=\langle 1/z, \lambda z;|\lambda|=1\rangle=S$, Therefore, we get  $E(z^{deg(P)})=E(z^2)$. Since the semigroup $Z$ acts on $S$ and $S\cap MP$ by semiconjugacy, the amenability of $E(z^2)$ and $E(z^2)\cap MP$ follows from Theorem \ref{th.Klawe} (see also Theorem 22 of \cite{CMAmenability}).

For part 3, it is enough to show that $E(z^2)$ contains a free subsemigroup of rank $2$. To show this we claim: 

\textit{The semigroup of affine maps $\Gamma:=\langle nz,  nz+\tau \rangle$ is free of rank $2$ for every $\tau\neq 0$ and $n\geq 2$}.

Indeed, by suitable conjugacy we can assume that $\Gamma$ is generated by $\gamma_1(z)=nz$ and $\gamma_2(z)=nz+1$. 
If there is a non-trivial relation, then we can assume it has the following form $$\gamma_1\circ \omega_1=\gamma_2\circ \omega_2,$$ where $\omega_1,\omega_2\in \Gamma$. Since the coefficients of $\omega_i$ are integers, the evaluation of this equation at $0$ gives a contradiction. 

Now, let $\tau=ia$ with $a$ irrational with respect to $2\pi$. Then the exponential map $E:\mathbb{C}\rightarrow \mathbb{C}^*$ defines, by semiconjugacy, a homomorphism $E^*:\langle nz,nz+\tau \rangle \rightarrow \langle z^n,e^{\tau}z^n\rangle \subset E(z^2)$ which is an isomorphism by the choice of $\tau$.

Part 4. Let $S=\langle U, T\rangle$, then $S$ is the semidirect product  of the abelian semigroups $U$ and $T$ since $T$ acts on  $U$  by semiconjugacy. Then $S$ is right amenable by Theorem \ref{th.Klawe}. Besides, $\lambda z \in U$ be an element of infinite order, and $z^n\in T$ with $n\geq 2$, then semigroup $\langle \lambda z^n,z^n\rangle$ is free by part 3.

\end{proof}

Since the maps $\lambda z^k$ and $z^l$ for $k,l\geq 2$ commute if and only if $\lambda$ is an appropiate root of unit, we note that part 3 of Theorem \ref{th.MP} also follows from Theorem \ref{th.ZieveZhan}. 

\begin{proof}[Proof of Theorem \ref{th.RIMLIM}]
If $RIM(S)\subset LIM(S)$, then $S$ is  amenable, and thus $S$ does not contain free subsemigroups of rank $2$ by Theorem \ref{th.ERNeumann}. 

Now let us check that $S$ is cancellative. Since $S$ is right cancellative, it is enough to show that it is left cancellative. Otherwise, there are   $a,x,y\in S$ satisfying   $$ax=ay$$ with  $a,x,y$  non-injective. By theorem 29 in \cite{CMAmenability}, there are numbers $k,m, n>0$ such that $a^k=x^m=y^n$. Hence, 

$$a^kx=x^{m+1}=xa^k$$ and similarly,
$a^ky=ya^k.$ So, if $ax=ay$, then  $xa^k=ya^k$, thus  $x=y$ by right cancellativity.

Let $S$ be a cancellative semigroup without free subsemigroups of rank $2$. By Theorem \ref{th.ERNeumann}, $S\subset E(R)$ for some $R\in S$. If again $S=S_0\sqcup S_1$ is the decomposition in the semigroups of injective and non-injective elements of $S$, then as above $S_0$ is a finite group, hence is amenable. Besides the semigroup $S_1$ is amenable by theorem 30 in \cite{CMAmenability}
 since, by cancellativity and Theorem \ref{th.Levinexc}, every pair of elements in $S_1$ share a common iteration. Therefore $S$ is amenable as it is the union of amenable subsemigroups.  

\end{proof}

Now we continue with the proof of Theorem \ref{th.three}.
\begin{proof}[Proof of Theorem \ref{th.three}] 
 Let $S$ be a left amenable semigroup of polynomials, then $S$ is an amenable semigroup of polynomials by \cite{PakovichAmenRat}. By  theorem 35 in \cite{CMAmenability}, there exists a polynomial $P$ and finite group $G$ of affine transformations such that:
 \begin{itemize}
  \item $G \subset E(P)$, 
  \item there exists a number $n$ so that every element of $G$ commutes with $P^n$, and
  \item $S\subset \langle G, P \rangle. $ 
 \end{itemize}
 
Then the semigroup $\langle G, P \rangle$ is cancellative and $P$ acts on $G$ by semiconjugacy as an automorphism, say $\rho$. Therefore  $\langle G, P \rangle\cong G\rtimes_\rho \langle P \rangle$ and is amenable by the second part of Theorem \ref{th.Klawe}. 

Assume that $LIM(\langle G, P \rangle)\subset RIM(\langle G,P \rangle)$, then by theorem 1 in \cite{CMAmenability} we have  $$LIM(\langle G, P \rangle)= RIM(\langle G,P \rangle).$$ Since $S$ is a cancellative and amenable subsemigroup of $\langle G, P\rangle$, then $LIM(S)=RIM(S)$ by applying corollary 30 in \cite{CMAmenability} to the semigroups $S\subset \langle G, P \rangle$  endowed with the antiproduct and we are done. 
 
Hence to finish the proof we need to show that  $LIM(\langle G, P \rangle)\subset RIM(\langle G,P \rangle)$. 

 Let $M\in L_\infty^*(\langle G, P \rangle)$ be a left invariant mean. First we show that $M$ is invariant under the right action of $G$ on $L_\infty(\langle G, P \rangle)$. Since the elements of $G$ commute with an iteration of $P$, say with $P^k$, then $P$ acts by semiconjugacy on $G$ as an automorphism of $G$. Hence $P$ induces a representation $\rho:\langle P \rangle \rightarrow Aut(G)$, so $\langle G, P \rangle$ is the semidirect product $G\rtimes_\rho\langle P \rangle$.

Let $l_h$ and $r_h$ be the left and right actions of $h$ on $L_\infty(\langle G, P \rangle)$, respectively. Let $$A_l(G)(\phi)=\frac{1}{card(G)} \sum_{g\in G} l_g(\phi)$$ be the left average of $G$ on $L_\infty(\langle G, P \rangle)$. Analogously, let $A_r(G)$ be the right average of $G$.  

We claim that $A_r(G)(\phi)=A_l(G)(\phi)$ for every $\phi\in L_\infty(\langle G, P \rangle)$. Indeed, for an element $s =h_s\circ P^{t_s}\in \langle G, P \rangle$ for a suitable $h_s\in G$ and a number $t_s\geq 0$, we have $$A_r(G)(\phi)(s)=\frac{1}{card(G)} \sum_{g\in G} r_g(\phi)(h_s\circ P^{t_s})$$
$$=\frac{1}{card(G)}\sum_{g\in G} \phi(h_s\circ P^{t_s}\circ g)=\frac{1}{card(G)}\sum_{g\in G} \phi(\rho( P^{t_s})(g)\circ h_s \circ P^{t_s})$$
$$=\frac{1}{card(G)}\sum_{g\in G}l_{\rho(P^{t_s})(g)}{\phi}(s)=A_l(G)(\phi)(s)$$
the last equality holds since $\rho(P^{t_s})$ is an automorphism of $G$. Which completes the claim.

By duality $A_r^*=A_l^*$. Since $M$ is left invariant, then 
$$M=A_l^*(M)=A_r^*(M).$$
Besides, for every $h\in G$ we have $A_r^*(G)\circ r^*_h=r^*_h\circ A_r^*(G)=A_r^*(G).$ Hence, $r_h$ leaves $M$ invariant for every $h\in G$.  

To finish the proof, it is enough to show that $M$ is invariant under the right action of $P$ on $L_\infty(\langle G, P \rangle)$. 

First we show that  \begin{equation}\label{eq} A_l(G) \circ l_P(\phi)=A_r(G)\circ r_P(\phi). \tag{*}\end{equation} Indeed, again  for $s=h_s\circ P^{t_s} $, for suitable $h_s\in G$ and $t_s\geq 0$ we compute  $$A_r(G)\circ r_P(\phi)(s)=\frac{1}{card(G)}\sum_{g\in G} \phi(h_s\circ P^{t_s}\circ P\circ g)$$ $$=\frac{1}{card(G)}\sum_{g\in G}\phi(\rho(P^{t_s+1})(g)\circ h_s \circ \rho(P)(h^{-1}_s)\circ P\circ h_s\circ P^{t_s}).$$ Since $G$ is cyclic and $\rho(P)$ is an automorphism of $G$, the latter is equal to

$$A_l(G)\circ l_P(\phi)(s).$$

By duality and equation \eqref{eq}, we have.

$$l^*_P \circ A^*_l(G)=r_P^*\circ A^*_r(G). $$

Since $M$ is a left invariant mean, by the equality above $M$ is also invariant under the right action of $P$, as we wanted to show.

\end{proof}

The following is an immediate corollary.

\begin{corollary}\label{cor.LIMRIM}
 If a non-exceptional semigroup $S$ of polynomials is amenable, then $LIM(S)=RIM(S)$.
\end{corollary}

\begin{proof}
By Theorem 1 in \cite{CMAmenability} we have that $RIM(S)\subset LIM(S)$, Theorem \ref{th.three} gives the opposite inclusion.
\end{proof}

\begin{proof}[Proof of Theorem  \ref{th.ptaktheorem}] Part 1. Assume that $E(P)$ is amenable, then $E(P)$ is cancellative by Corollary \ref{cor.amenable} and hence, $E(P)$ is embeddable into a group by  Theorem \ref{th.Ore}.  Now we will apply Pt\'ak's construction. According to Theorem \ref{th.Ptak}, since $E(P)$ is embeddable there exists a finitely generated free group $F$ and a normal subgroup $N\subset F$ such that $E(P)$ is embeddable into the group $F/N$. By the Uniformization Theorem we can choose the groups $N\subset F$ as discrete subgroups of $Aut(\Delta)\simeq PSL(2,\mathbb{R})$, where $\Delta$ is the unit disc. Thus the suitable Riemann surface is  $\Delta/N$. 
 
 Reciprocally, assume $E(P)$ is embeddable into a group. If $E(P)=E_0\sqcup E_1$ is the decomposition of injective and non-injective elements of $E(P)$, then $E_0$ is a group which is right amenable by Corollary \ref{cor.charrigham}, thus it is amenable. The semigroup $E_1$ is amenable by Part 2 of Theorem \ref{th.RIMLIM} as $E_1$ is cancellative and every rank $2$ subsemigroup of $E_1$ is right amenable by Corollary \ref{cor.charrigham}. So $E(P)$ is amenable since it is the union of amenable semigroups. 

 Part 2. If $E(P)$ is abelian, then it is cancellative as it is right cancellative. By Theorem \ref{th.Ore} it is embeddable into a finitely generated abelian virtually cyclic group. Since $E_0$ is isomorphic to a finite rotation group, we are done.
 
 Reciprocally, if $E(P)$ is embeddable into the M\"obius group, then $E(P)$ is an amenable cancellative semigroup by theorem 35 in \cite{CMAmenability} and there exist a polynomial $Q\in E(P)$, the group $E_0=E(P)\cap Mob$  and a homomorphism $\rho: \langle Q  \rangle \rightarrow Aut(E_0)$, such that $E(P)=E_0\rtimes_\rho \langle Q \rangle$ where $Q$ is considered acting by semiconjugacy on $E(P)$. Then $E(P)$ is embeddable into a virtually cyclic group which is a semidirect product of a group isomorphic to $E_0$ and an infinite cyclic group inside the M\"obius group. Then $E(P)$ must be abelian.
\end{proof}

\begin{proof}[Proof of Theorem \ref{th.embeddable}]

Part 1. Let $G_+=\{g_i(z)=\lambda_i z^{n_i}\}$ be a countable set of generators of $S$. Let $\Gamma_+=\{\gamma_i(z)=n_i z +\tau_i\}$, where $\tau_i=Ln(\lambda_i)$ with $Ln$  the principal branch of natural logarithm.
 The exponential map defines, by semiconjugacy, a homomorphism $h$ from the semigroup $AS=\langle \Gamma_+\rangle$ onto $S$ such that:
 $$\exp\circ \gamma = h(\gamma)\circ \exp, $$ for every $\gamma\in AS$. 

Note that if the restriction of $h$ on a semigroup $U<AS$  is injective, then the semigroup $h(U)$ is embeddable into a group.

Let $\Gamma=\langle AS\rangle$ be the group generated by $AS$,  which is also generated by $\Gamma_+$.  Denote by $D:g\rightarrow deg(g)$ the degree character on $MP$ and $T$ be the group generated by $\{\gamma_i\circ (z+2\pi i)\circ \gamma_i^{-1}|\gamma_i\in \Gamma_+ \}$. 

Let $N\subset \Gamma$  be the normal subgroup defined as follows:
$$N=\begin{cases}
Id, \textnormal{ if } \langle z+2\pi i \rangle \cap \Gamma=Id \\
T\cap \Gamma,  \textnormal{ otherwise.}
\end{cases}$$ 

Note that every element $\tau \in N$ has the form $$\tau(z)=z+\frac{p}{q} 2\pi i$$ where $p$ and $q$ are integers and $q\in D(S)$.

If $\Pi:\Gamma \rightarrow \Gamma/N$ is the projection, then the correspondence $\phi=\Pi \circ h^{-1}:S\rightarrow \Gamma/N$ is a homomorphism. 

We claim that $\phi$ is injective. Indeed, assume that $\phi:S\rightarrow \Gamma$ is not injective, then there are distinct elements $g_1,g_2 \in S$ such that $\phi(g_1)=\phi(g_2)$. 
 
 Let $\tau_1\in h^{-1}(g_1)\cap \Gamma$ and $ \tau_2\in h^{-1}(g_2) \cap \Gamma$, then   $\tau_2=\tau_1+ \frac{p}{q}2\pi i$ for a suitable $\frac{p}{q}\in \mathbb{Q}$, $q\neq 1$.  Let $g_3\in S$ satisfying $D(g_3)=q$, then $g_3\circ g_1=g_3 \circ g_2$ contradicting the cancellativity of $S$. 

 Part 2. By Theorem \ref{th.ZieveZhan} either $S$ is free or there are numbers $k,l,m\geq 1$ so that $R^k\circ Q^l$ commutes with $R^m$. If $R$ is non-exceptional, then by Theorem \ref{th.ERNeumann}, these maps and hence also $Q$ and $R$ share the same measure of maximal entropy,
so $S$ is right amenable by Theorem \ref{th.ERNeumann} and hence is embeddable by Theorem \ref{th.Ore}. But,  a finitely generated free semigroup is embeddable.
 
 If  $R$ is conjugated to an element of $MP$, then by  Theorem \ref{th.ZieveZhan}, we have  $S\subset MP$  and thus $S$ is embeddable by Part 1.
 
In \cite{JuliaMem}, Julia showed that if two rational maps $Q$ and $R$ commute then necesarily these maps share the same Julia set. Then again if $R$ is conjugated to a Tchebichev polynomial by Theorem \ref{th.ZieveZhan}, theorem 1 in \cite{AtelaHu} and Julia theorem, then $Q$ is also conjugated to either a Tchebichev map $T$ or $-T$. So, the semigroup $S$ is  nearly abelian and hence is right amenable by Proposition \ref{Pr.RX}, thus $S$ is embeddable by Theorem \ref{th.Ore}.
As Fatou sets are empty for Latt\`es examples, we are done with the exceptional case.

 Part 3. Since a right amenable semigroup is embeddable into a group by Theorem \ref{th.Ore}, then $S=\ast_i F_i$ is a free product of embeddable semigroups $F_i$. Let $\Gamma_i$ be the groups in which $F_i$ embedds, then $S$ is embeddable into the free product of the groups $\Gamma_i$.
 
\end{proof}

\begin{proof}[Proof of Theorem \ref{th.leftKoopman}]
Fix a probability measure $\sigma$ on the Riemann sphere $\overline{\C}$. Let $C(\overline{\C})$ be the space of complex continuous functions with the supremum norm. The correspondence $H:C(\overline{\C})\rightarrow L_\infty(S)$ given by $$H(\phi)(g)=\int_{\overline{\C}}K_g(\phi) d\sigma$$ is a continuous linear map, where $g\in S$ and $\phi\in C(\overline{\C}).$
Let $X=cl(Image(H)) \subset L_\infty(S)$, then $X$ is invariant by the left action of $S$ on $L_\infty(S)$. 
Since the characteristic function $\chi_{\overline{\C}}\in C(\overline{\C})$ is a fixed point of $K_g$ for every $g\in S$, then  $\chi_S\in X$. By assumption there exists a mean $F$ on $L_\infty(S)$ which is invariant under the  left action of $S$ on $X$, then $F\not\equiv 0$ on $X$. Let $f\in C^*(\overline{\C}) $ be the functional given by $$f(\phi)=F(H(\phi)).$$
We claim that $f$ is $K_h$-invariant for every $h\in S$. Indeed, for  $h\in S$, 

$$H(K_h(\phi))(g)=\int_{\overline{\C}} K_g(K_h(\phi))d\sigma=\int_{\overline{\C}} K_{h\circ g}(\phi)d\sigma$$
$$=H(\phi)(h\circ g)=l_h(H(\phi))(g),$$ where
$l_h:L_\infty(S)\rightarrow L_\infty(S)$, $l_h(\psi)(g)=\psi(h\circ g)$ is the left action of $h$ on $L_\infty(S)$. Since $F$ is left invariant on $X$, we have

$$f(K_h(\phi))=F(H(K_h(\phi)))=F(l_h(H(\phi)))=F(H(\phi))=f(\phi)$$ as claimed.

By the Riesz representation theorem, there exists an $S$-invariant probability measure $\mu$ as required.\end{proof}

\begin{proof}[Proof of Proposition  \ref{pr.semigroupwithout}] Let $R$ be a hyperbolic rational map and $U$ be a backward fundamental set for $\langle R \rangle$.  If $\mu$ is a finite invariant measure for  $R$, then  $supp(\mu)$ can not be contained in $
U$. Let $Q$ be another hyperbolic rational map with a backward fundamental set $V$ satisfying the Klein combination  condition with $R$ and $U$. Then, the semigroup $\Gamma=\langle Q, R \rangle$ does not admits a common finite invariant measure and it is free by Theorem \ref{th.HM}. Hence, by Theorem \ref{th.leftKoopman}, $\Gamma$ is not Koopman left amenable.  
\end{proof}

So, the question of whether a rank $2$ free  semigroup of non-exceptional rational maps admits a non-atomic invariant measure is connected with Open question \ref{conjetura1}.

\begin{proof}[Proof of Proposition \ref{pr.LdM}]
As mentioned in the introduction, this proposition follows from theorem 1.2 in  \cite{BakerdeMarco} and theorem 1.5 in \cite{yuanzhang}.

\end{proof}

To prove Theorem \ref{th.Theorem5} we need the following Lemma.

\begin{lemma}\label{Lm.Teo6}
The semigroup $S=\langle \lambda z, z^k \rangle$ is right amenable for every natural number $k$.
\end{lemma}
\begin{proof} We can assume that $k>1$ and $\lambda \neq 0$.
 Note that every element $g\in S$ has the form $\gamma\circ r$ where $\gamma$ is an element of the cyclic semigroup $\langle \lambda z \rangle $ and $r\in \langle z^k \rangle$. Since $z^k$ acts by semiconjugacy on $\langle \lambda z \rangle $ as an endomorphism, then $S$ is a semidirect product of $\langle \lambda z \rangle$ and $\langle z^k \rangle$, then $S$ is right amenable by Theorem \ref{th.Klawe}. 
\end{proof}

Now we proceed to the proof of Theorem \ref{th.Theorem5}.

\begin{proof}[Proof of Theorem \ref{th.Theorem5}] 
Let $\mu$ be a  non-discrete invariant measure for $S$ with $supp(\mu)\subset F(P)$ for $P\in S$. By the classification of periodic components of $F(P)$, the support of $\mu$  is a subset of Siegel discs cycles $Z_i$ up to a finite number of finite invariant sets (see \cite{QHDSul}).  As the support of $\mu$ is an invariant compact set then  $supp(\mu)\subset \bigcup_i (\bigcup^{k(i)}_{j=1} D_i^j)\subset \bigcup Z_i$ such that 
\begin{itemize}
 \item  $D_i^j\subset Z_i$, for $j=1...k(i)$, are closed discs where $P^{k(i)}(D_i^j)=D_i^j$ for all $i,j$ and $\partial D_i^j$ are leaves of the invariant foliation in the Siegel discs. 
\item  $\partial D_i^j \subset supp(\mu)$. 
\end{itemize}

 So $P$ defines an automorphism of $\bigcup\bigcup \partial D_i^j$.  Hence,  for every component $U$ of $\bigcup D_i^j$ there is an $n$ such that $\mu$ restricted to $U$ is a measure $\mu_U$  invariant for $P^n$.

Fix an arbitrary $Q\in S$. We claim that there exists a component $U$ of $\bigcup D_i^j$ and number $m$ such that the measure $\mu_U$ is invariant for $Q^m$. Indeed, since the support of $\mu$ is a compact invariant set for $Q$, then by the principle of maximum, the interior of every $D^j_i$ is a subset of $F(Q).$ 
If $\partial U \cap F(Q)\neq \emptyset$ for a suitable component $U$ of $\bigcup D_i^j$ then $\partial U$ is a leaf of an invariant foliation in a Siegel disc cycle in $F(Q)$, thus there exists $m$ as claimed. Otherwise,  
$\partial U \subset J(Q)$ for every component $U$ of $\bigcup D_i^j$, then every $U$ is a component of $F(Q)$. Fix any such $U$ and let $U_0$ be a periodic component in the orbit of $U$, then $\mu$ restricted to $U_0$ is an invariant measure for a suitable $Q^m$ as claimed.

Next step is to show that the semigroup $\langle P^n,Q^m\rangle$ is abelian. Let $U$ be a component as in the previous claim. By the desintegration theorem we can assume that $supp(\mu)=\partial U$,  since $\partial U$ is a complex analytic curve, then the map  $h:\partial U \rightarrow \mathbb{S}^1$  linearizing $P^n$ is conformal with $h\circ P^n \circ h^{-1}(z)=\lambda z$, where $|\lambda|=1$ and is not a root of unit. The map $h$ sends $\mu$ on $\partial U$ to the one dimensional Lebesgue measure on $\mathbb{S}^1$.  Then  $h\circ Q^m\circ h^{-1}:\mathbb{S}^1\rightarrow \mathbb{S}^1$  is an holomorphic  map leaving the one dimensional Lebesgue measure invariant. Hence $h\circ Q^m \circ h^{-1}(z)= \tau z^k$ for suitable $k\geq 1$ and $|\tau|=1$. 

Let us show that $k=1$, and hence in this case the semigroup $\langle P^n, Q^m \rangle$ is conjugated to the abelian semigroup $\langle \lambda z, \tau z\rangle$ as claimed.

Otherwise, if $k>1$ there exists $x$ a repelling fixed point for $\tau z^k$. Since  $\langle P^n,Q^m \rangle $ is conjugated to   $\langle \lambda z, z^k \rangle$, then it is right amenable by Lemma \ref{Lm.Teo6}. Therefore $P^n$ and $Q^m$ share the same measure of maximal entropy by Theorem \ref{th.ERNeumann}, then $P$ and $Q$ share the same measure of maximal entropy. Hence $F(P)=F(Q)$ and, thus, the repelling fixed point $h^{-1}(x)$ belongs to $F(Q)$ which is a contradiction. Hence $k=1$.

By Eremenko-Ritt theorem, $P$ and $Q$ share a common iteration. But $Q$ is arbitrary, then Theorem 1 in \cite{CMAmenability} applies and finishes the proof. 
\end{proof}

Now we proceed with the proof of Theorem \ref{th.polynomiallike}.

\begin{proof}[Proof of Theorem \ref{th.polynomiallike}]
First let us discuss the exceptional case. If either one of $Q$ and $R$, say $R$, is exceptional, then $R$ is either conjugated to $z^n$ or to a Chebyshev polynomial, since  Latt\'es maps are postcritically finite with empty Fatou set, Latt\'es maps do not admit polynomial-like measures. 

If $R$ is affinely conjugated to $z^n$, then we can assume that $R(z)=z^n$ and $\mu$ is the one-dimensional Lebesgue measure on the unit circle. In this case, $Q(z)=\alpha z^k$ for a suitable $\alpha$ with $|\alpha|=1$, $k\geq 1$. Since there exists  $z_0$, a common periodic point, $\alpha$ is a root of unit. But $z^n$ acts on the cyclic semigroup $\langle Q \rangle$, and hence the semigroup $\langle Q, R \rangle$ has the structure of a semidirect product with $\langle Q \rangle$ as a factor. Then by Theorem \ref{th.Klawe}, the semigroup $\langle Q, R \rangle$ is right amenable. 

Similar arguments apply in the Chebyshev case. Indeed, if $R$ is Chebyshev, then by the theorem 1 in \cite{AtelaHu} either $Q$ is Chebyshev or $-Q$ is Chebyshev. In the first case, $\langle Q, R\rangle$ is abelian, and thus is amenable. If $-Q$ is Chebyshev, again $R$ acts on $\langle Q \rangle$ then $\langle Q, R\rangle$ is a semidirect product and we again apply Theorem \ref{th.Klawe}.

Now assume that neither $Q$ nor $R$ are exceptional. By hypothesis, there exists $n$ such that the  map $R^n$  defines a marked polynomial-like map $T:D_0\rightarrow D_1$ given by the restriction $T=R^n|_{D_0}$, with connected $J(T)$ and periodic point $z_0$. Let $P$ be the unique monic polynomial in the hybrid class of $T$ and let $\phi:D_0\rightarrow \C$ be the quasiconformal map conjugating $T$ with $P$. Let $\psi$ be the B\"ottcher coordinate of $P$ sending the basin of attraction at $\infty$ to the unit disc. Then the map $h=\psi\circ \phi$  conjugates $R$ with $z^d$ for a suitable $d$, and transforms the measure $\mu$ to the one-dimensional Lebesgue measure  on the unit circle. There exists an $m$ such that map $q=h\circ Q^m\circ h^{-1}$ extends to an analytic map on a neighborhood of the unit circle, leaving the one dimensional Lebesgue measure invariant.   Hence $q(z)=\alpha z^k$, for a suitable $k\geq 1$ and $|\alpha|=1$. Since $h(z_0)$ is a common periodic point for $p$ and $q$,  there are commuting iterations of $p$ and $q$. 

Since neither $Q$ nor $R$ are exceptional, by theorem 1 in \cite{Erefunc}, $Q$ and $R$ share an iteration,  then the semigroup $\langle Q, R \rangle$ is amenable by Theorem 31 in \cite{CMAmenability} which finishes the proof.

\end{proof}

To prove Theorem \ref{th.free}, we need the following proposition.
\begin{proposition}\label{prop.FustPrimRat}
 Let $P$ and $Q$ be two polynomials sharing the same measure of maximal entropy but not sharing a common iteration.   Furthermore, assume  $P$ is a primitive non-exceptional polynomial with $deg(P)>1$ and that there exists a common invariant probability measure $\mu$   for $P$ and $Q$, then $\mu$ is either discrete or coincides with the measure of maximal entropy for $P$ and $Q$. 
 
 \end{proposition}

\begin{proof}  By assumption $P$ is conjugated to $z^n+c$ where $n\geq 2$ and $c\neq 0$, then $E(P)$ is generated by $Deck(P)$ and the cyclic semigroup $\langle P \rangle$. Hence there exists a number $j\geq 1$ and $\gamma \in Deck(P)$ such that $Q=\gamma\circ P^j$ and $\gamma\neq Id$. If a probability measure $\mu$ is invariant for both $P$ and $Q$, then $\mu$ is invariant for $\gamma$ and thus invariant under $Deck(P)$ since it is cyclic. 

Let us show that either $\mu$ discrite or it is the measure of maximal entropy for $P$.   As $\mu$ is invariant for $Deck(P)$, and $Deck(P)$ acts transitively on $P^{-1}(z)$, then for $\phi \in C(\overline{\C})$  we have
$$\int \frac{1}{deg(P)} \sum_{i=1}^{deg(P)} \phi(\zeta_i)d\mu=\int L_P(\phi)d\mu=\int \phi d\mu. $$
Where $\zeta_i$ are branches of $P^{-1}$ and $L_P$ is the Lyubich operator for $P$.  Therefore $\mu$ is either discrete or the measure of maximal entropy for $P$ and $Q$ by the main theorem in \cite{LyubichErgodicTheory}.\end{proof}

\begin{proof}[Proof of Theorem \ref{th.free}]
Follows from theorem 4 in \cite{CMAmenability} and Proposition \ref{prop.FustPrimRat}.
\end{proof}

  \bibliographystyle{amsplain}
\bibliography{workbib}
\Addresses
\end{document}